\documentclass[11pt,leqno]{article}
\usepackage{amsmath,graphicx,amstext, amsfonts, amsthm, amsxtra, amssymb, verbatim}
\usepackage[mathscr]{euscript}

\begin{document}
%\input{notations.tex}
%%%%%%%%%%%%%%%%%%%%%%%%%%%%%%%%%%%%%%%%%%%%%%%%%%%%%%5
%-----------------------------------------------------------------------
\def\Sp{{\rm Sp}}
\def\Ra{{\rm Ra}}
\def\GM{{\rm GM}}

\def\per{{\rm X}}      % period map/matrix
\def\perr{{\sf q}}        %period matrix.....
\def\perdo{{\cal K}}   %period domain
\def\sfl{{\mathrm F}} %Space of filtrations
\def\sp{{\mathbb S}}  %Sphere

\newcommand\diff[1]{\frac{d #1}{dz}} %Differential operator
\def\End{{\rm End}}              %Endomorphism group

\def\sing{{\rm Sing}}            %The set of singularities
\def\spec{{\rm Spec}}            %The spectrume
\def\cha{{\rm char}}             %Charracteristic
\def\Gal{{\rm Gal}}              %The Galois group
\def\jacob{{\rm jacob}}          %the Jacobian ideal
\def\tjurina{{\rm tjurina}}      %the tjurina ideal
\newcommand\Pn[1]{\mathbb{P}^{#1}}   %Projective space of dimension #1
\def\Ff{\mathbb{F}}                  %Finite field
\def\Z{\mathbb{Z}}                   %Integer  numbers
\def\Gm{\mathbb{G}_m}                 %The multiplicative group
\def\Q{\mathbb{Q}}                   %Rational  numbers
\def\C{\mathbb{C}}                   %Complex numbers
\def\O{{\cal O}}                     %ring of integers of a number field
\def\as{\mathbb{U}}                  %Some affine space
\def\ring{{\mathsf R}}                         %A ring
\def\R{\mathbb{R}}                   %real numbers
\def\N{\mathbb{N}}                   %natural numbers
\def\A{\mathbb{A}}                   %affine space C^n
\def\uhp{{\mathbb H}}                %upper half plane
\newcommand\ep[1]{e^{\frac{2\pi i}{#1}}}% unipotent numbers
\newcommand\HH[2]{H^{#2}(#1)}        %Hodge structures
\def\Mat{{\rm Mat}}              %Matrices
\newcommand{\mat}[4]{
     \begin{pmatrix}
            #1 & #2 \\
            #3 & #4
       \end{pmatrix}
    }                                %two by two matrices
\newcommand{\matt}[2]{
     \begin{pmatrix}                 % one by two matrix
            #1   \\
            #2
       \end{pmatrix}
    }
\def\ker{{\rm ker}}              %kernel
\def\cl{{\rm cl}}                %Chern class
\def\dR{{\rm dR}}                %The subindex dR standing for de Rham
                                     %cohomology.

\def\hc{{\mathsf H}}                 %The set of Hodge cycles.
\def\Hb{{\cal H}}                    %Hodge bundle
\def\GL{{\rm GL}}                %The liner group
\def\pese{{\sf P}}                  %Period set
\def\pedo{{\cal  P}}                  %Period domain
\def\PP{\tilde{\cal P}}              %the period domain/ discrete group
\def\cm {{\cal C}}                   %the set of CM Hodge structures
\def\K{{\mathbb K}}                  %Field representing R or C
\def\k{{\mathsf k}}                  %Arbitrary field
\def\F{{\cal F}}                     %Hodge filtration bundle
\def\M{{\cal M}}
\def\RR{{\cal R}}
\newcommand\Hi[1]{\mathbb{P}^{#1}_\infty}%the hyperplane at infinity
\def\pt{\mathbb{C}[t]}               %Polynomials in t
\def\W{{\cal W}}                     %weight filtration
\def\gr{{\rm Gr}}                %graded pieces
\def\Im{{\rm Im}}                %imaginary
\def\Re{{\rm Re}}                %Real
\def\depth{{\rm depth}}
\newcommand\SL[2]{{\rm SL}(#1, #2)}    %SL(2,Z)
\newcommand\PSL[2]{{\rm PSL}(#1, #2)}  %PSL(2,Z)
\def\Resi{{\rm Resi}}              %Residue

\def\L{{\cal L}}                     %The moduli of polarized lattices in a
                                     %fixed vector spaces.
\def\Aut{{\rm Aut}}              %Automorphism group of a vectorspace
\def\any{R}                          %Any subring of the field of complex
                                     %numbers.
\newcommand\ovl[1]{\overline{#1}}    %Conjugation of #1.

\def\T{{\cal T }}                    %Tangent space
\def\tr{{\mathsf t}}                 %Transposition of matrices
\newcommand\mf[2]{{M}^{#1}_{#2}}     %New modular functions
\newcommand\mfn[2]{{\tilde M}^{#1}_{#2}}     %New modular functions
\newcommand\bn[2]{\binom{#1}{#2}}    %Binomial
\def\ja{{\rm j}}                 %j of a two by two matrix
\def\Sc{\mathsf{S}}                  %Simple cycles
\newcommand\es[1]{g_{#1}}            %Eisenstein series
\newcommand\V{{\mathsf V}}           %Milnor vector space
\newcommand\WW{{\mathsf W}}          %Similar to Milnor vector space
\newcommand\Ss{{\cal O}}             %Structural sheaf
\def\rank{{\rm rank}}                %rank of a module
\def\Dif{{\cal D}}                   %Differentials
\def\gcd{{\rm gcd}}                  %greatest common divisor
\def\zedi{{\rm ZD}}                  %zero divisors of a module
\def\BM{{\mathsf H}}                 %Brieskorn module
\def\plf{{\sf pl}}                             %Picard-Lefschetz formula
\def\sgn{{\rm sgn}}                      %sign
\def\diag{{\rm diag}}                   %diagonal matrix
\def\hodge{{\rm Hodge}}
\def\HF{{\sf F}}                                %The hodge filtration of the brieskon module
\def\WF{{\sf W}}                               %The weight filtration of the brieskon module
\def\HV{{\sf HV}}                                %humbert variety
\def\pol{{\rm pole}}                               %pole divisor
\def\bafi{{\sf r}}
\def\codim{{\rm codim}}                               %codimension
\def\id{{\rm id}}                               %identity
\def\gms{{\sf M}}                           %Gauss-Manin system
\def\Iso{{\rm Iso}}                           %Gauss-Manin system

\def\hl{{\rm L}}    %holomorphic limit
\def\bcov{{\rm BCOV}}
\def\imF{{\rm F}}
\def\imG{{\rm G}}

\def\ord{{\rm ord}}
\newcommand\FF{F(a_1,\cdots,a_n|z)} 
\newcommand*\pFqskip{8mu}
\catcode`,\active
\newcommand*\pFq{\begingroup
        \catcode`\,\active
        \def ,{\mskip\pFqskip\relax}%
        \dopFq
}
\catcode`\,12
\def\dopFq#1#2#3#4#5{%
        {}_{#1}F_{#2}\biggl(\genfrac..{0pt}{}{#3}{#4};#5\biggr)%
        \endgroup
}

\def\CC{\tilde \Z}

%------------------------------------------------------------------------
\def\losu {{G}}
\newcommand\dwork[1]{\delta_{#1}}
%%%%%%%%%%%%%%%%%%%%%%%%%%%%%%%%%%%%%%%%%%%%%%%%%%%%%%%%5
%---------------------------------------------
\newtheorem{theo}{Theorem}
\newtheorem{exam}{Example}
\newtheorem{coro}{Corollary}
\newtheorem{defi}{Definition}
\newtheorem{prob}{Problem}
\newtheorem{lemm}{Lemma}
\newtheorem{prop}{Proposition}
\newtheorem{rem}{Remark}
\newtheorem{conj}{Conjecture}
\newtheorem{calc}{}

\begin{center}
{\LARGE\bf Automorphic forms for triangle groups: Integrality properties
%\footnote{The text is under construction. Any comment is welcome.} 
%and differential equations for modular forms
}
\\
\vspace{.25in} {\large {\sc Hossein Movasati and Khosro M. Shokri }}\footnote{
Instituto de Matem\'atica Pura e Aplicada, IMPA,
Estrada Dona Castorina, 110,
22460-320, Rio de Janeiro, RJ, Brazil,}

%E-mail:
%{\tt hossein@impa.br} \\
%{\tt www.impa.br/$\sim$hossein}
\end{center}
\begin{abstract}
We classify all primes  appearing in the denominators of the Hauptmodul $J$ and modular forms  for  non-arithmetic triangle groups with a cusp.
These primes have a congruence condition in terms of the order of the generators of the group. As a corollary
we show that for the Hecke group of type $(2,m,\infty)$, the prime $p$ does not appear in the denominator of
$J$ if and only if $p\equiv \pm 1\pmod m$. 
\end{abstract}
%{\tiny

\section{Introduction}
The theory of automorphic forms for Fuchsian groups was first developed by Poincar\'{e}. 
His construction based on series carrying nowadays his name, analogous to the classical Eisenstein series (Fuchsian theta-series in his terminology).
A disadvantage of this method is that explicit $q$-expansions which are fruitful part of the theory
of modular forms for arithmetic groups are not available for these groups.
An alternative approach with concentrating on explicit $q$-expansions for a special case, namely hyperbolic triangle groups is available.
Here we briefly explain this method (for details see \cite{hokh2}).\\
Let us consider the Halphen system
\begin{equation}
\label{tatlis}
\left \{ \begin{array}{l}
\dot t_1=
( a-1)(t_1t_2+t_1t_3-t_2t_3)+( b+ c-1)t_1^2\\
\dot t_2 =
(b-1)(t_2t_1+t_2t_3-t_1t_3)+( a+ c-1)t_2^2\\
\dot t_3 =
( c-1)(t_3t_1+t_3t_2-t_1t_2)+(a+ b-1)t_3^2
\end{array} \right. 
\end{equation}
with 
\begin{equation}
 \label{abc}
1-a-b=\frac{1}{m_1}, \quad 1-b-c=\frac{1}{m_2}, \quad 1-a-c=\frac{1}{m_3}=0, 
\end{equation}
and $m_1\leq m_2\in\N\cup\{\infty\}$ with the hyperbolicity condition $\frac{1}{m_1}+\frac{1}{m_{2}}<1$. 
Here $\dot t=q\frac{dt }{d q}$ and we consider $t_i\in\C[[q]]$ as formal power series in $q$
with the initial condition:
$$t_1(0)=t_3(0)=0,$$
\begin{equation*}
t_2=\left \{ \begin{array}{l}
-1-(m_1+1)q+O(q^2) \quad \qquad \qquad \qquad \qquad \text{if} \quad m_2=\infty\\
              -1+(m_1^2m_2+m_1^2-m_1m_2^2-m_2^2)q+O(q^2) \quad \text{otherwise.}
              \end{array}\right.
\end{equation*}
The recursion of Halphen system determines uniquely
$t_i$'s.
If we set $q=\exp(\frac{2\pi i\tau}{h})$, where $h= 2\cos(\frac{\pi}{m_1})+2\cos(\frac{\pi}{m_2})$, then $t_i$'s are meromorphic functions
on  $\Im (\tau)> \tau_0$ for some real positive $\tau_0$. Now,  rescaling $q$ by  multiplying a constant,  
$t_i$'s become meromorphic on the whole upper half plane  with modular property with respect to the 
triangle group  $\Gamma_{\mathfrak{t}}:=\langle \gamma_1,\gamma_2,\gamma_3\rangle\subset \SL 2\R$ of type $\mathfrak{t}=(m_1,m_2,\infty)$, where
%We remind that as an abstract group, a triangle group has presentation $\langle g_1,g_2,g_3\mid g_i^{m_i}=1=g_1g_2g_3\rangle$.
%Given one such triangle group, one can find another by conjugating by any $g\in\PSL 2\R$.
%The triangle group of a given type $(m_1,m_2,m_3)$ is unique up to this conjugation, and so is determined up to 3 real
%parameters. As the automorphic functions of $\Gamma$ and $g\Gamma g^{-1}$ are
%related by $f(\tau)\leftrightarrow f(g^{-1}\tau)$, we are free to choose any realization.  Of course,
%this conjugation will in general affect the integrality of Fourier coefficients, so some choices
%are better than others. We choose the following generators
\begin{equation}\label{trigen}\gamma_1=\left(\begin{matrix}2\cos(\frac{\pi}{ m_1})&1\cr -1&0\end{matrix}\right)\,,
\gamma_2=\left(\begin{matrix}0&1\cr -1&2\cos(\frac{\pi} {m_2})\end{matrix}\right)\,,
\gamma_3=\left(\begin{matrix}1&h\cr 0&1\end{matrix}\right)
\end{equation}
%for the group $\Gamma_{\mathfrak{t}}$, with the property
$$
\gamma_1\gamma_2 \gamma_3=\gamma_1^{m_1}=\gamma_2^{m_2}=-I_{2\times 2}.
$$
The Hauptmodul for this triangle  group  is given by
\begin{equation}
\label{11sep}
J= \frac{t_3-t_2}{t_3-t_1}.
\end{equation}
We define
\begin{align}
E_{2k}^{(1)}&:=(t_1-t_2)(t_3-t_2)^{k-1}\in 1+q\Q[[q]],\\
\label{ek2}
E_{2k}^{(2)}&:=(t_1-t_2)^{k-1}(t_3-t_2)\in 1+q\Q[[q]].
\end{align}
In \cite{hokh2} we showed that  the algebra of automorphic forms for the group $\Gamma_{\mathfrak{t}}$
 with $m_1\leq m_2<\infty$ is generated by
$$
E_{2k}^{(1)}, \quad 3\leq k\leq m_1, \qquad E_{2k}^{(2)}, \quad 2\leq k\leq m_2,
$$
and when $m_1<m_2=\infty$, the algebra is generated by
$$
E_{2k}^{(1)}, \quad 1\leq k\leq m_1.
$$
For the triangle group of type  $(\infty,\infty,\infty)$ see \cite{hokh2}. 
The coefficients of $J$ are rational numbers apart from the rescaling (transcendental) constant.
 This constant appears to fit the convergence of $J$ in the whole upper half plane 
(see \cite{wo83} for a proof of transcendence of this constant).
The rationality comes out from the recursion of the Halphen system for the coefficients of $t_i$. 
A natural question would be a classification of primes which appear in the denominators. 
In \cite{hokh2}  we stated a conjecture concerning this problem.
The aim of this article is to give a complete answer to this question. 
We recall that a power series $f$ is called $p$-integral if, after multiplication of $f$ by a constant , its 
coefficients are  $p$-adic integers.
We say an algebra of power series in $\Q[[q]]$ is $p$-integral if
it has a basis with $p$-integral elements. We say an object (function or algebra) is 'almost' integral if it is $p$-integral for all but finitely many $p$.
\begin{theo}
\label{29oct}
\label{maintheo}
Let $m_1\leq m_2\in \N$ and $p$ be a prime with $p>2m_1m_2$.  
The Hauptmodul  $J$, defined in \eqref{11sep}, for the triangle group of
type $(m_1,m_2,\infty)$ is $p$-integral if and only if for some $\epsilon=\pm 1$ and $\epsilon '=\pm 1$ we have 
$$
\left( p\stackrel{2m_1}{\equiv}\epsilon,\, p\stackrel{2m_2}{\equiv}\epsilon'\epsilon \right )  \hbox{ or }  
\left ( p\stackrel{2m_1}{\equiv} m_1+\epsilon,\, p\stackrel{2m_2}{\equiv}m_2+\epsilon '\epsilon  
\right ). 
$$
For the triangle group $(m,\infty,\infty)$ and $p>2m$ the Hauptmodul  $J$ is $p$-integral if and only if
$$
p \stackrel{2m}{\equiv}\pm 1.
$$
%$$
%\begin{array}{ll}
%p \stackrel{2m}{\equiv}\pm 1 &  \hbox{ if $m$ even }\\
%p \stackrel{m}{\equiv}\pm 1 &  \hbox{ if $m$ odd }\\
%\end{array}
%$$
\end{theo}

We need the condition $p>2m_1m_2$ for the 'if' part of the theorem and the 'only if' part only requires only that $p$ does not divide $2m_1m_2$.
Some computations show that the theorem must be valid with this weak hypothesis on $p$. For example for $(2,5,\infty)$, our 
experimental computations shows that the $J$ function up to $183$ terms is $p$-integral for $p=11,19$ (see  below, Corollary \ref{hecke}).
\begin{coro}
 \label{arithmatic}
The Hauptmodul $J$ for a triangle group is  almost integral if and only if 
\begin{align*}
(m_1,m_2,\infty)=&(2,3,\infty),\, (2,4,\infty),\, (2,6,\infty),\, (2,\infty,\infty),\, (3,3,\infty),\,(3,\infty,\infty)\\
&(4,4,\infty),(6,6,\infty),(\infty,\infty,\infty).
\end{align*}
\end{coro}
This is the Takeuchi's classification in \cite{tak77} of arithmetic triangle 
groups with a cusp and of type $(m_1,m_2,\infty)$. For explicit uiformizations of modular curves attached to these 9 cases see \cite{bayer}.
%Such a group appears as  the monodromy group of the Gauss hypergeometric equation.
\begin{coro}
\label{hecke}
Let $3\leq n\in \N$. For a prime $p>4n$ the Hauptmodul $J$ of the Hecke group $\Gamma_{(2,n,\infty)}$ 
is $p$-integral if and only if $ p\equiv \pm 1 \pmod n$.
\end{coro}
We remind that Corollary \ref{hecke} was a conjecture made by Leo Garret in his PhD thesis \cite{leo}. He proved some partial results
in this direction. Precisely, he showed that if $p\equiv 1 \mod {4n}$, then $J$ is $p$-integral.
\begin{coro}
\label{basis}
Let $p>2m_1m_2$ be a prime number.  The algebra of automorphic forms for the triangle group of type $(m_1,m_2,\infty)$ is $p$-integral if and only if $p$ 
satisfies the conditions of Theorem \ref{maintheo}.
\end{coro}

Integrality  problem for the coefficients of modular forms for noncongruence subgroups of $\Gamma(1)=\SL 2\Z(=\Gamma_{(2,3,\infty)}$)
was a task in \cite{scholl}.  There,  Scholl proves that  there exist  positive integers $d$ and $N$ such that
$d^na_n\in\mathcal{O}_{F}[\frac{1}{N}]$, where $a_n$ is the  $n$-th Fourier coefficient of a modular form of weight
$k\in \frac{1}{2}\Z$ for some subgroup of $\Gamma(1)$ and $F$ a number field.
A conjecture of Atkin and Swinnerton-Dyre predicts that $N=1$ if and only if the subgroup contains a congruence subgroup 
(see \cite{atkin}).
The result of Scholl implies that at most finitely many distinct primes can appear in the denominators of modular forms for 
a noncongruence subgroup of $\Gamma(1)$. On the other
hand, when the group is not commensurable with $\Gamma(1)$, one would expect infinitely many primes in the denominators.
This prediction is compatible with our result in the case of hyberpolic triangle groups. 

The paper is organized in the following way. 
In \S\ref{dwork} we introduce the main technique for establishing the results of the paper. 
This is namely the Dwork method which is based on a Lemma due to Dieudonn\'e and a Theorem due to Dwork,
see \cite{dwo73, dw94}. %In \S\ref{comalg} we  present a lemma  whose proof requires an explicit  primary decomposition of an ideal. 
In \S\ref{proof} we prove the corollaries and give the proof of the main theorem.  

The second author would like to thank CNPq-Brazil for financial support and IMPA for its lovely research ambient.

%%%%%%%%%%%%%%%%%%%%%%%%%%%%%%%%%%%%%%%%%%%%%%%%5
\section{Dwork method}
\label{dwork}
The main idea in the proof of Theorem \ref{29oct} is based on the Dwork method. Here we briefly review this method. 
\subsection{Dwork map}
\label{dworkmap}
\def\CC{\tilde \Z}
%For a rational number $a$, we write $r=\frac{r_1}{r_2}$, where $r_1\in \Z$ and $r_2\in\N$ are respectively the nominator and denominator of $r$ and so
% $(r_1,r_2)=1$.   %We also denote by $(a \mod b)$ the integer $0\leq x\leq b-1$ such that $a \equiv_b x$. 
For the $p$-adic integers $\Z_p$, the Dwork map $\dwork{p}: \Z_p\to \Z_p$ is given by
$$
 x=\sum_{s=0}^\infty x_sp^s \longmapsto 1+\sum_{s=0}^\infty x_{s+1}p^s,\ \ \ 0\leq x_s\leq p-1.
$$
In other words, for every $x$, with $x\equiv x_0\pmod {p\Z_p}$,  $\dwork{p}(x):=1+\frac{x-x_0}{p}$.
Denote by  $\Z_{(p)}$ the set of $p$-integral rational numbers.
% which are not in $\Z^{-}\cup\{0\}$.
We have a natural embedding $
\Z_{(p)}\hookrightarrow \Z_p$. 
The map $\dwork{p}$ leaves $\Z_{(p)}$ invariant because for $x\in \Z_{(p)}$, $\dwork{p}(x)$ is the unique number such that
$p\dwork{p}(x)-x \in \Z\cap [0,p-1]$. 
For rational numbers there exists an alternative definition for  the Dwork map as follows.
Let  $x=\frac{x_1}{x_2},$ with $x_1$ and  $x_2>0$ integers and a prime $p$ which does not divide $x_2$, we have
\begin{equation}
\label{121212}
\dwork{p}(x):=\frac{p^{-1}x_1 \hbox{ mod } x_2}{x_2},
\end{equation}
where $p^{-1}$ is the inverse of $p$ mod $x_2$ (note that $x_1$ and $x_2$ may have common factors). 
The denominators of $x$ and $\dwork{p}(x)$ are the same and 
$\dwork{p}(1-x)=1-\dwork{p}(x)$.  
%\label{10dec2012}
 For any finite set of rational numbers, there is a finite decomposition of prime numbers such that in each class the 
function $\dwork{p}$ is independent of
the prime $p$. Indeed for the set of  primes $p \stackrel{x_2}{\equiv} r$,  $\dwork{p}(x)$ only depends on $x$ and $r$. 
%As an special case we have the following equivalence:
%\begin{equation}
 %\label{deltapp'}
%\dwork{p}(x)+\dwork{p'}(x)=1 \Longleftrightarrow p+p'\equiv 0 \pmod {x_2}.
%\end{equation}
%We use this property in the proof of Lemma \ref{1july2013} and Theorem \ref{maintheo}.
\subsection{Gauss hypergeometric function}
Let us consider the following hypergeometric differential operator
\begin{equation}
 \label{7july13}
L:\theta^2-z(\theta+a)(\theta+b),
\end{equation}
with $\theta= z\frac{d}{dz}$ and 
\begin{equation}
\label{tarifab}
a=\frac{1}{2}(1-\frac{1}{m_1}+\frac{1}{m_2}), \ \
b=\frac{1}{2}(1-\frac{1}{m_1}-\frac{1}{m_2}),
\end{equation}
where $2 \leq m_1,m_2\in \N\cup\{\infty\}$ and $\frac{1}{m_1}+\frac{1}{m_2}<1$. Note that these $a,b$ are slightly 
different from those in the introduction. From now on we will only use (\ref{tarifab}). 
The monodromy group of the corresponding differential equation is the triangle group  of type $(m_1,m_2,\infty)$, see for instance 
\cite{hokh2}.
The Frobenius basis of \eqref{7july13} around $z=0$ is given by
$\{F(z), F(z)\log z+G(z)\}$, where
\begin{align}
F(a,b|z)&=1+\sum_{i=1}^\infty A_i(a,b)z^i=1+\sum_{n=1}^\infty \frac{(a)_n(b)_n}{n!^2}z^n,\\
G(a,b|z)&=\sum_{i=1}^\infty B_i(a,b)z^i=
\sum_{n=1}^\infty \frac{(a)_n(b)_n}{n!^2}(\sum_{i=0}^{n-1}\frac{1}{a+i}+\frac{1}{b+i}-\frac{2}{1+i})z^n.
\end{align}
Let us define
\begin{equation}
\label{defDq}
 D(a,b|z):=\frac{G(a,b|z)}{F(a,b|z)} ,\ \ \ \ q(a,b|z):=z\exp (D(a,b|z)),
\end{equation}
and $D$ is called the Schwarz map. The Hauptmodul $J$ introduced in the Introduction is given by
$$
J=\frac{1}{z(\kappa\cdot q)},\ \ \ \kappa:=-2m_1^2m_2^2,
$$
where $z(q)$ is the inverse of $q$ as a function in $z$, for more details see \cite{hokh2}.

%It is well-known that the image of the Schwarz function 
%\marginpar{\tiny add some explanation about the connection of this with the Halphen system and what is the exact formula of 
%J in this interpretation also for proposition 1 is needed}
%$D(a,b|z):=\frac{G(a,b|z)}{F(a,b|z)}$
%is the upper half plane $\uhp$ and the inverse of $q(a,b|z):=z\exp (D(a,b|z))$  is a holomorphic function from $\uhp$ to $\C$.

\subsection{Dwork's Theorem}
The following lemma is the additive version of Dieudonn\'{e}-Dwork lemma and 
frequently is used in the proof of $p$-integrality of power series.
\begin{lemm}
\label{dieudonne}
Let $u(z)\in z\Q_p[[z]]$. Then $\exp(u(z))\in 1+z\Z_p[[z]]$, if and only if 
$$
\exp(u(z^p)-p\,u(z))\in 1+p\Z_p[[z]].
$$
\end{lemm}
For a more general statement and the proof  see \cite{dw94}, p.54.
The following theorem is the main part of Dwork method. 
\begin{theo}
\label{9/11}
Let $D$ be the Schwarz map, defined in \eqref{defDq} and $p$ a prime number coprime with $2m_1m_2$. We have
$$
D(\dwork{p}(a),\dwork{p}(b)|z^p))\equiv pD(a,b|z) \pmod {p\Z_p[[z]]}.
$$
\end{theo}
As a remark we mention that the original Dwork's theorem is valid  not only for arbitrary $a,b\in \Z_p$ but also for
generalized hypergeometric series. For a proof see \cite{dwo73}.  
\begin{coro}
\label{25nov}
If
 \begin{equation}
\label{12oct2013}
 \{\dwork{p}(a),\dwork{p}(b)\}=\{a,b\} \hbox{  or   }  \{1-a,1-b\},
\end{equation}
holds then $q(a,b|z)$ is $p$-integral. 
\end{coro}
\begin{proof}
For  $\{\dwork{p}(a),\dwork{p}(b)\}=\{a,b\}$, the statement is an immediate consequence of Dwork's theorem and Lemma \ref{dieudonne}
with $u(z)=D(a,b|z)$ and  the fact that $\ord_p(n!)<n$.
For the second case, thanks to the Euler identity 
$$
F(a,b|z)=(1-z)^{1-a-b}F(1-a,1-b|z),
$$
one can easily check that the logarithmic solution of \eqref{7july13} and so  $G$  satisfy the same identity.
Then the result follows from the first case.
\end{proof}
Corollary \ref{25nov} gives a sufficient condition for $p$-integrality of $q(a,b|z)$. In order to proof Theorem \ref{maintheo} we need also a
necessary condition. The following corollary is an step toward this goal. 
\begin{coro}
\label{nopint}
Let $p$ and $q(a,b|z)$ as before. If the function $q(a,b|z)$ is $p$-integral, then
\begin{equation}
\label{4dec2012}
D(\dwork{p}(a),\dwork{p}(b) |z)\equiv D(a,b|z) \pmod {p\Z_p[[z]]},
\end{equation}
and vice versa.
\end{coro}
\begin{proof}
If $q(a,b|z)$ is $p$-integral, from Lemma \ref{dieudonne} we have
$$
D(a,b|z^p)-pD(a,b|z)=\log (1+p\, h(z)),
$$
for some $h(z)\in z\Z_p[[z]]$. But 
$$
\log(1+p\,h(z))=\sum_{n=1}^\infty (-1)^n\frac{p^nh(z)^n}{n}\in p\,z\Z_p[[z]]. 
$$
Hence
$$D(a,b|z^p)\equiv pD(a,b|z)\pmod {p\Z_p[[z]]}.$$
Combining with the congruence of Theorem \ref{9/11} the result follows.
The other side is similar.
\end{proof}

%From this corollary one can see that if for example $\{\dwork{p}(a),\dwork{p}(b)\}=\{a,b\}$ then $q$ is $p$-integral.

In the continuation we will  determine  complete conditions such that the congruency  \eqref{4dec2012} holds. We will prove that it is equivalent to
(\ref{12oct2013}) in Corollary \ref{25nov}.  
\begin{lemm}
\label{31may2013}
Let $\k$ be a field of characteristic $p\not=2$ and $a_1,a_2,b_1,b_2\in\k$.  The coefficients of $z^i,\ i=1,2$ 
in  \begin{equation}
%\label{7julio}
D(a_2,b_2|z) \hbox{ and } D( a_1,b_1|z)
\end{equation}
are equal 
if and only if
\begin{equation}
\label{9d2012}
\{a_2,b_2\}=\{a_1,b_1\} \text{ or } \{1-a_1,1-b_1\}.
\end{equation}
\end{lemm}
Note that in general $G(a_1,b_1|z), \ a_1,b_1\in\k$ (consequently $D(a_1,b_1|z)$), is not well-defined because in its expression we have division by primes.
However, it makes sense to talk about the coefficients of $z$ and $z^2$ in characteristic $p\not=2$. 

\begin{proof}
 Let $\sigma=a+b,\ \tau=ab$. The coefficients of $D(a,b|z)$ can be written in terms of the symmetric polynomials $\sigma,\tau$.
Let $C_{k}(\sigma,\tau)$ be the $k$-th coefficient of $D(a,b|z)$. By definition we have $C_1(\sigma,\tau)=\sigma-2\tau$, so the assumption implies that
\begin{equation}
 \label{29nov1}
\sigma_1-2\tau_1\stackrel{p}{\equiv} \sigma_2-2\tau_2.
\end{equation}
Now for $C_2(\sigma,\tau)$ we have
\begin{align*}
4C_2(\sigma_1,\tau_1)&-4C_2(\sigma_2,\tau_2)=\sigma_1^2-5\sigma_1\tau_1+5\tau_1^2-\sigma_2^2+5\sigma_2\tau_2-5\tau_2^2+\sigma_1-\tau_1-\sigma_2+\tau_2\\
&= (\sigma_1-2\tau_1)^2-(\sigma_2-2\tau_2)^2+\tau_2(\sigma_2- 2\tau_2)-\tau_1(\sigma_1- 2\tau_1)+\tau_2^2-\tau_1^2+\tau_1-\tau_2\\
&\equiv (\tau_2-\tau_1)(\sigma_1-2\tau_1+\tau_1+\tau_2-1)\pmod{ p}
\end{align*}
In the above we have used the congruence \eqref{29nov1} in the last line. Hence from the last congruence we conclude that 
either $\tau_1\stackrel{p}{\equiv}\tau_2$
or $\sigma_1\stackrel{p}{\equiv} \tau_1-\tau_2+1$. The first (second) case together with the equation \eqref{29nov1} gives the first 
(second) possibility mentioned in \eqref{9d2012}.
\end{proof}

\section{Proofs}
\label{proof}
In this section we give a proof of Theorem \ref{29oct} and its corollaries announced in the Introduction. 
 %%%%%%%%%%%%%%%%%%%%%%%%%%%%%%%%%%%%%%%%%%%%%%%%%%%%%%%%%%%%%%%%%%%%%%%%%%%%%%%5
\subsection{Proof of Theorem \ref{29oct}}
First we show that for $p>2m_1m_2$, the $p$-integrality of $q(a,b|z)$
is equivalent to  condition \eqref{12oct2013}. In fact
if $q(a,b|z)$ is $p$-integral, then equation \eqref{4dec2012} holds. In particular we can apply Lemma \ref{31may2013} for 
$\{a_1,b_1\}=\{a,b\}$, $\{a_2,b_2\}=\{\dwork{p}(a),\dwork{p}(b)\}$ and the finite field $\k=\frac{\Z}{p\Z}$. It follows that 
$\{\dwork{p}(a),\dwork{p}(b)\}$ congruent to $\{a,b\}$ or $\{1-a,1-b\}$ modulo $\k$. But $p>2m_1m_2$, in particular it is greater than
the denominators of $a,b<1$. Since the action of $\dwork{p}$ does not
change the denominator, so the  above congruence is indeed an equality
in $\Z$.
Hence the only thing to complete the proof is to show that the conditions 
of Theorem \ref{29oct} are equivalent to equation \eqref{12oct2013}. In odrer to do this we analyze the equality \eqref{12oct2013} case by case.\\
Recall that
$$
a=\frac{a_1}{a_2}=\frac{m_1m_2-m_1+m_2}{2m_1m_2},\  b=\frac{b_1}{b_2}=\frac{m_1m_2-m_1-m_2}{2m_1m_2}.
$$
1. $\dwork{p}(a)=a,\ \dwork{p}(b)=b$ or  $\dwork{p}(a)=1-a,\ \dwork{p}(b)=1-b$.
By definition of the Dwork map in \eqref{121212}, in this case we have
$p^{-1}a_1\stackrel{a_2}{\equiv} \epsilon a_1$ and 
$p^{-1}b_1\stackrel{b_2}{\equiv}\epsilon b_1$, where $\epsilon=1$ corresponds to the first case and $\epsilon=-1$ corresponds to the second case.
Since $p$ is odd, the above congruences are equivalent to 
\begin{align}
p(m_1+m_2)&\equiv\epsilon(m_1+m_2)\pmod{2m_1m_2}\notag\\
\label{15oct}
p(m_1-m_2)&\equiv\epsilon(m_1-m_2)\pmod{2m_1m_2}.
\end{align}
Once adding and subtracting of congruences in \eqref{15oct} we find that $p\stackrel{m_i}{\equiv}\epsilon$ for both $i=1,2$.
From this fact one can easily check that \eqref{15oct} is equivalent to  
$$
\left( p\stackrel{2m_1}{\equiv}\epsilon,\, p\stackrel{2m_2}{\equiv}\epsilon \right )  \hbox{ or }  
\left ( p\stackrel{2m_1}{\equiv} m_1+\epsilon,\, p\stackrel{2m_2}{\equiv} m_2+\epsilon  
\right ). 
$$
2. The case $\dwork{p}(a)=b$ and $\dwork{p}(b)=a$ or $\dwork{p}(a)=1-b,\ \dwork{p}(b)=1-a$. 
Again by definition of the Dwork map we have 
$p^{-1}a_1\stackrel{a_2}{\equiv} \epsilon b_1$ and 
$p^{-1}b_1\stackrel{b_2}{\equiv}\epsilon a_1$, , where $\epsilon=1$ corresponds to the first case and $\epsilon=-1$ corresponds to the second case.
Like the previous case these congruences are equivalent to
\begin{align}
p(m_1+m_2)&\equiv\epsilon(m_1-m_2)\pmod{2m_1m_2}\notag\\
\label{15oct3}
p(m_1-m_2)&\equiv\epsilon(m_1+m_2)\pmod{2m_1m_2},
\end{align}
and one can check that this is equivalent to
$$
\left( p\stackrel{2m_1}{\equiv}\epsilon,\, p\stackrel{2m_2}{\equiv}-\epsilon \right )  \hbox{ or }  
\left ( p\stackrel{2m_1}{\equiv} m_1+\epsilon,\, p\stackrel{2m_2}{\equiv} m_2-\epsilon  
\right ). 
$$
Now for $(m,\infty,\infty)$, from \eqref{tarifab} we see that $a=b=\frac{m-1}{2m}$. Then  Condition \eqref{12oct2013} is equivalent to
$p(m-1)\stackrel{2m}{\equiv}\epsilon(m-1)$, canceling $m-1$ and the fact that $p$ is odd proves the statement.

%%%%%%%%%%%%%%%%%%%%%%%%%%%%%%%%%%%%%%%%%%%5
\subsection{Proof of Corollary \ref{arithmatic}}
We see that if one of $m_i$, $i=1,2$,  does not belong to the set $\{2,3,4,6,\infty\}$, 
then there is a residue like $r\neq \epsilon, m_i+\epsilon$ with $(r,2m_i)=1$. Then by Dirichlet theorem there are infinitely many primes $p\stackrel{2m_i}{\equiv} r$
and by Theorem \ref{29oct}, $J$ is not $p$-integral for such primes, which is a contradiction. 
Now checking all possibilities we find the list given
in the statement of the corollary.
\subsection{Proof of Corollary \ref{hecke}}
We note that the first condition in Theorem \ref{29oct}, namely $p$ modulo $2m_1$ automatically holds for $m_1=2$
and every prime greater than $3$. Hence $J$ for Hecke group $\Gamma_{(2,n,\infty)}$ is $p$ integral if and only if
$p\stackrel{2n}{\equiv}\pm 1$ or $n\pm 1$. This is equivalent to $p\stackrel{n}{\equiv}\pm 1$.

\subsection{Proof of Corollary \ref{basis}}
Let $m_2$ be  finite (the case $m_2=\infty$ resolves in a similar way). Let also $\mathfrak m$ be 
the algebra of automorphic forms for $\Gamma_{\mathfrak t}$. Using Theorem \ref{maintheo}, it is enough to 
prove that $J$ is $p$-integral if and only  if $\mathfrak m$ is $p$-integral.
Let $E_4=E_{4}^{(2)}$ and 
$E_6=E_{6}^{(2)}$, where $E_k^{(2)}$ are defined in \eqref{ek2}. We have
$$
J=\frac{E_{4}^3}{E_{4}^3-E_{6}^2}.
$$
If $J$ is not $p$-integral then one of the functions $E_4$ or $E_6$ is not $p$-integral and hence $\mathfrak m$ is not $p$-integral ($E_4$ and $E_6$ are members of this algebra and 
we use the convention that the $p$-integrality property is defined up to multiplication by a constant).
Now from the Halphen system one can check that 
$$
t_1-t_2=\frac{\dot J}{J},\quad t_3-t_2=\frac{\dot J}{J-1},
$$
and so
$$
E_{2k}^{(1)}=\frac{J-1}{J}(\frac{\dot J}{J-1})^k,\ \ 
E_{2k}^{(2)}=(\frac{\dot J}{J})^k\frac{J}{J-1}.
$$
If $J$ is $p$-integral then  all its derivatives are
$p$-integral, an so, all the above elements are $p$-integral. A subset of these functions form a basis for $\mathfrak m$.
Note that if an algebra $\mathfrak m$ is $p$-integral and we have a basis $A$ of $\mathfrak m$ then after multiplication of 
the elements of $A$ by proper constants, $A$ turns to be a basis of $\mathfrak m$ with $p$-integral elements.   

%%%%%%%%%%%%%%%%%%%%%%%%%%%%%%%%%%%%5
\subsection{Final remarks}
We expect that Theorem \ref{maintheo} to be true for primes $p$ less than and  coprime to $2m_1m_2$. 
This is equivalent to say that 
Corollary \ref{25nov} is "if and only if".  If $q(a,b|z)$ is $p$-integral, then 
$q(\dwork{p}^n(a),\dwork{p}^n(b)|z)$ is $p$-integral  for all $n\in \N$ and hence 
we can use Corollary \ref{nopint} and Lemma \ref{31may2013} and conclude that 
\begin{equation}
 \{\dwork{p}^n(a),\dwork{p}^n(b)\}\stackrel{p} \equiv\{a,b\} \hbox{  or   }  \{1-a,1-b\}.
\end{equation}
This does not seem to be sufficient in order to conclude the true equality. In order to further investigate the 
$p$-integrality of $q(a,b|z)$ we need 
to use more data  from the congruency (\ref{4dec2012}).

Unfortunately, in the literature there are no applications  for the $q$-expansion of automorphic forms for non-arithmetic triangle groups.
The main reason is the lack of Hecke theory for such automorphic forms. The rationality of coeffients is
one of the main obstacles. We saw that for a class of primes the integrality in the level of $p$-adic integers holds.
One question here is whether this integrality has distinguished enumerative properties. 

The hypergeometric functions $F,F\log(z)+G$ up to some $\Gamma$-factors are periods of the following family of curves
$$
C^{a,b,c}_{z}: y=x^a(x-1)^b(x-z)^c,
$$
where $a,b,c$ are give in \eqref{abc}. Another interesting problem is to find a geometric description for the result
of Theorem \ref{maintheo} using the above family of curves. 
This curve and its Jacobian, are extensively studied by Wolfart et al. 
in connection with the algebraic values of the  Schwarz map, see \cite{wolf2, wo83} and references therein.

Another interesting problem which we would like to address here is the $p$-integrality  
for generalized hypergeometric equations of order $n$ whose 
local exponents at $z=0$ are all zero. As we mentioned before, Dwork's theorem is valid in this general case. 
The Gauss hypergeometric equation corresponds to $n=2$. We  obtain in a similar way $p$-integrality results for the mirror map 
(the analog of $q(a,b|z)$ for arbitrary $n$).
For $n>2$ in the  absence of the Euler identity, the only sufficient condition for $p$-integrality of the mirror map
is that $\dwork{p}$ acts as a permutation on the local exponents of the differential equation at $\infty$ (an analog of Corollary \ref{25nov}).
Then an interesting question is the converse, as we did in this article for $n=2$. 
An important situation with applications in algebraic geometry and mathematical physics is the case in which 
the mirror map is almost integral. 
Then a simple observation shows that
in Lemma \ref{nopint}, the congruence \eqref{4dec2012} is indeed an equality. In \cite{Roq}, the author, using differential Galois theory,
showed that, this equality holds if and only if $\dwork{p}$ acts as a permutation for almost all $p$ 
%(it means that essentially there is no non-trivial identity for the quotient of two different hypergeometric equation with rational parameters).
This fact establishes the problem of classification of
all hypergeometric equations with maximal unipotent monodromy and  with integral mirror map.
As a corollary for $n=4$, which is important in mirror symmetry, the well-known $14$ cases is obtained. 

As a final remark we would like to mention the possible relationship between  almost integrality and 
arithmeticity of the mondromy groups. The coincidence of the Takeuchi list and the list of Corollary \ref{arithmatic} 
does not seem to be 
casual.
However the intrinsic connection between these two differnet worlds is not yet clear. Despite the existence of Dwork's method  
for non unipotent cases, this method does not determine the rest of the Takeuchi list, i.e., arithmetic tiangle groups without cusp.
The situation for $n>2$ is more obscure. For example for $n=4$ even in the case with maximal unipotent monodromy
it has been shown that among the $14$ cases some of them are arithmetic and some of them are not (see for instance \cite{sive}). 
For a nice discussion in this subject we refer the reader to \cite{sar}.

{\tiny

\def\cprime{$'$} \def\cprime{$'$} \def\cprime{$'$}

}

%\bibliography{biblio}

%\bibliographystyle{plain}

\end{document}